\documentclass[11pt]{article}
\usepackage{amsmath,amssymb,amsthm}
\usepackage{graphicx,subfigure,float,url}
\usepackage{pdfsync}
\usepackage[english]{babel}
\usepackage[utf8]{inputenc}
\usepackage{enumitem}
\usepackage{fancybox}
\usepackage{listings}
\usepackage{mathtools}
\usepackage{amsfonts}
\usepackage[bottom]{footmisc}
\usepackage{verbatim}

\usepackage{float}
\usepackage{makeidx}
\normalbaselines

\newtheorem{Theorem}{Theorem}[section]

\newtheorem{Lemma}[Theorem]{Lemma}

\newtheorem{Corollary}[Theorem]{Corollary}

\newtheorem{Remark}[Theorem]{Remark}

\newcommand{\RR}{{{\rm I} \kern -.15em {\rm R} }}

\newcommand{\C}{{{\rm l} \kern -.42em {\rm C} }}

\newcommand{\nat}{{{\rm I} \kern -.15em {\rm N} }}

\newcommand{\be}{\begin{equation}}
\newcommand{\ee}{\end{equation}}
\newcommand{\beq}{\begin{eqnarray}}
\newcommand{\eeq}{\end{eqnarray}}
\newcommand{\beqs}{\begin{eqnarray*}}
\newcommand{\eeqs}{\end{eqnarray*}}
\newcommand{\bt}{\begin{Theorem}}
\newcommand{\et}{\end{Theorem}}
\newcommand{\br}{\begin{Remark}}
\newcommand{\er}{\end{Remark}}
\newcommand{\bc}{\begin{Corollary}}
\newcommand{\ec}{\end{Corollary}}
\newcommand{\bl}{\begin{Lemma}}
\newcommand{\el}{\end{Lemma}}
\newcommand{\bd}{\begin{definition}}
\newcommand{\ed}{\end{definition}}

\renewcommand{\geq}{\geqslant}

\renewcommand{\leq}{\leqslant}

\title{Exponential decay for nonlinear abstract evolution equations with a countably infinite number of time-dependent time delays}
\author{
Alessandro Paolucci\footnote{Dipartimento di Ingegneria e Scienze dell'Informazione e Matematica, Universit\`{a} di L'Aquila, Via Vetoio, Loc. Coppito, 67010 L'Aquila Italy (\texttt{alessandro.paolucci2@graduate.univaq.it}).}
}

\date{}

\begin{document}

\textwidth=160 mm

\textheight=225mm

\parindent=8mm

\frenchspacing

\maketitle

\begin{abstract}
In this paper we analyze a nonlinear abstract evolution equation with an infinite number of time-dependent time delays and a Lipschitz continuous nonlinear term. By using a fixed point argument we prove the existence of a mild solution. This allows us to take into account also nonnegative time delays. Furthermore, by using  Gronwall estimates, exponential decay of the solution is also proved under some smallness assumptions on the parameters appearing in the system and on the initial data. Finally some examples are illustrated. 
\end{abstract}

\vspace{5 mm}

\def\qed{\hbox{\hskip 6pt\vrule width6pt
height7pt
depth1pt  \hskip1pt}\bigskip}



\section{Introduction}
\label{pbform}
%\hspace{5mm}

\setcounter{equation}{0}
In the last decades time delay systems have been studied by several authors. Time delay effects are very common in physical and biological models and it's well-known that the introduction of a delay term in the model can induce instability (see e.g. \cite{Datko2, Datko, NP}). From a mathematical point of view, it's very useful studying the stability behavior of solution to such systems, which is the main topic of a large number of papers (see, for instance, \cite{AC, ANP, Dafermos, Guesmia, KP, NP2015, NP2018, NP, PP, P, Yang, YY}). In this paper, we consider $H$ an Hilbert space endowed with inner product $\langle \cdot, \cdot\rangle_H$ and norm $||\cdot||_H$ and let $U:[0,+\infty)\to H$ satisfy the following abstract equation:
\begin{equation}\label{abstract}
\begin{array}{l}
\displaystyle{ U'(t)=AU(t)+\sum_{i=1}^{+\infty} k_i(t)B_iU(t-\tau_i(t))+F(U(t)), \quad t\geq 0,}\\
\displaystyle{U(0)=U_0,}\\
\displaystyle{B_iU(s)=g_i(s), \quad s\in I:=\bigcup_{i=1}^{+\infty} [-\tau_i(0),0], \ i\in\nat.}
\end{array}
\end{equation}
In \eqref{abstract}, $A$ generates an exponentially stable $C_0$-semigroup $\{S(t)\}_{t\geq 0}$, namely there exist $M,\omega >0$ such that
\begin{equation}\label{semigroup}
||S(t)||_{\mathcal{L}(H)} \leq Me^{-\omega t}, \quad t\geq 0,
\end{equation}
where $\mathcal{L}(H)$ is the set of all linear operators from $H$ into itself. We will show that the system will be exponentially stable if we add also a countably number of delay terms. For any $i\in \nat$, we consider $\tau_i(\cdot):[0,+\infty)\to [0,+\infty)$ the time-dependent time delays. We suppose that
\begin{equation}\label{derivatadelay}
\tau_i'(t)\leq c_i<1,
\end{equation}
where $c_i$ are real constants, for any $i\in\nat$. Therefore, for any $t\geq 0$, $t-\tau_i(t)\geq -\tau_i(0)$. We define
\begin{equation}\label{maxdelay}
\tau_*:=
\begin{cases}
\max_{i\in\nat} \tau_i(0), \ \text{if it exists}, \\
\sup_{i\in\nat} \tau_i(0)\leq +\infty, \ \text{otherwise}.
\end{cases}
\end{equation}
In the first case, $I=[-\tau_*,0]$, while in the second one the initial conditions are defined on the interval $I=(-\tau_*,0]$. Moreover, $B_i:H\to H$ are bounded linear operators from $H$ into itself, for any $i\in\nat$. In order to simplify the notation, we set
$$
b_i:=||B_i||_{\mathcal{L}(H)}, \qquad \forall i\in \nat.
$$ 
The coefficients $k_i:[0,+\infty)\to\RR$ are functions in $L^1_{loc}([0,+\infty))$ for any $i\in\nat$. Furthermore, we take the nonlinear term $F:H\to H$ Lipschitz continuous, namely there exists $L>0$ such that
$$
||F(U)-F(V)||_H\leq L||U-V||_H,
$$
for any $U,V\in H$. We suppose $F(0)\equiv 0$. Finally, we take $g_i\in \mathcal{C}(I;H)$ for any $i\in \nat$. 

We are interested in studying the well-posedness and the exponential stability result for system \eqref{abstract} under a smallness assumption on the delays and on the initial data, which we explain later on. Here, we stress the fact that zero delays are also allowed. This is due to the fact that we will use a fixed point argument in order to study the existence and uniqueness of solution to \eqref{abstract}, instead of the usual iterative one. Indeed, as described in \cite{YY}, the iterative argument allows us to study the well-posedness of solution to a certain system by using time interval of length $\tau_{inf}:=\inf_{i\in\nat} \inf_{t\geq 0} \tau_i(t)$, which, in our case, can be equal to $0$ for some $t\geq 0$.

The theory presented in this paper extends the one introduced in \cite{KP}. There, the authors considered system \eqref{abstract} with a finite number of time-varying time delays and they studied the exponential decay of its solution. They also presented some examples, such as (both linear and nonlinear) damped or viscoelastic wave equations with delay feedback. See also \cite{NP2018} for the study of abstract evolution equations with a finite number of time delays in the nonlinear source term and \cite{YY} for evolution equations with infinite constant time delays.

Our theory deals with many models, studied by several authors. Stability of delayed linear and nonlinear wave equation have been studied in many papers (see for instance \cite{ANP, FP, Guesmia, P2012}). We also mention \cite{Chentouf, MK, PP, SS, Yang} for the study of plate equations, also named either Euler-Bernoulli or Timoshenko model, and \cite{AC, ACS, ANP, PP, P} for viscoelastic wave-type equation. In particular, we refer to \cite{A} for a Timoshenko model for beams in one dimension.

\textbf{Main assumption.} Throughout the paper, we assume the following inequality: 
\begin{equation}\label{assumption}
 e^{\omega \tau_*}\sum_{i=1}^{+\infty}\int_{-\tau_*}^0  \frac{1}{1-c_i} |k_i(\varphi_i^{-1}(z))| \cdot ||g_i(z)||_H dz+\sum_{i=1}^{+\infty} \int_0^{+\infty} \frac{b_i}{1-c_i}|k_i(\varphi_i^{-1}(z))| dz <\frac 1 M ,
\end{equation}
where $\varphi_i(s)=s-\tau_i(s)$, for any $i\in\nat$ and $s\in[-\tau_*,+\infty)$. Notice that $\varphi_i(\cdot)$ is invertible for any $i\in\nat$ thanks to hypothesis \eqref{derivatadelay} on the delay terms.

We will prove existence and uniqueness of a mild solution to \eqref{abstract}. Here, we stress the fact that this result is given also if the semigroup generated by the operator $A$ in \eqref{abstract} is increasing, namely if it satisfies for some $M,\omega >0$
$$
||S(t)||_{\mathcal{L}(H)} \leq Me^{\omega t}, \quad t\geq 0.
$$
Therefore, we will prove the following theorem.
\begin{Theorem}\label{well-posedness}
Suppose that $A$ is the generator of a $C_0$-semigroup $\{ S(t)\}_{t\geq 0}$, which satisfies the following inequality:
\begin{equation}\label{dis1}
||S(t)||_{\mathcal{L}(H)}\leq M e^{\omega t}, \qquad M,\omega>0.
\end{equation}
Then, for any $U_0\in H$ and for any $g_i\in \mathcal{C}(I;H)$, $i\in\nat$, satisfying \eqref{assumption}, there exists a unique mild solution to \eqref{abstract} satisfying
\begin{equation}\label{mild_solution}
U(t)=S(t)U_0+\int_0^t S(t-s)\left[ \sum_{i=1}^{+\infty} k_i(s)B_iU(s-\tau_i(s))+F(U(s))\right] ds,
\end{equation}
for any $t\geq 0$.
\end{Theorem}
As already mentioned, the main result of this paper is the following theorem.
\begin{Theorem}\label{decay}
Let $U:[0,+\infty)\to H$ be the unique mild solution to \eqref{abstract} where $A$ generates an exponentially stable $C_0$ semigroup $\{ S(t)\}_{t\geq 0}$, namely $S(t)$ satisfies \eqref{semigroup}. If
$
L<\omega,
$
and assumption \eqref{assumption} is satisfied, then we have that
\begin{equation}
||U(t)||_H\leq M\left(   ||U_0||_H+e^{\omega \tau_*}\sum_{i=1}^{+\infty}\int_{-\tau_*}^0 \frac{1}{1-c_i} |k_i(\varphi_i^{-1}(s))|\cdot ||g_i(s)||_H ds \right)e^{1-(\omega-L) t},
\end{equation}
for any $t\geq 0$.
\end{Theorem}
Differently from Theorem \ref{well-posedness}, in Theorem \ref{decay} the exponential stability assumption \eqref{semigroup} on the semigroup $\{ S(t)\}_{t\geq 0}$ is needed.

The rest of the paper is organized as follows. In Section \ref{Wellpos} we will prove the existence of a mild solution (i.e. a Duhamel's formula) to system \eqref{abstract} using a fixed point argument. Then, in Section \ref{stabsec} we will prove Theorem \ref{decay}, i.e. the exponential stability of the unique mild solution to \eqref{abstract}. Finally, in Section \ref{example} we will present some technical examples of wave equation and plate system with viscoelastic and strong damping terms.
\section{Well-posedness result}\label{Wellpos}
In this Section we will prove the existence and uniqueness of a mild solution to \eqref{abstract}. As mentioned in the previous Section, we will use a fixed point argument, similarly to \cite{YY}.

\begin{proof}[Proof of Theorem \ref{well-posedness}]
Let us consider $\omega'>\omega>0$, which we will fix later on, and let us define the following Banach space:
$$
Y:=\left\{ U(t)\in \mathcal{C}([0,+\infty);H) \ : \ \sup_{t\geq 0} \ e^{-\omega' t} ||U(t)||_H <+\infty\right\},
$$
where
$$
||U||_Y:= \sup_{t\geq 0} \ e^{-\omega' t} ||U(t)||_H.
$$
We define the functional $\Phi :Y\to Y$ such that for any $U(t)\in Y$
$$
\Phi U(t) =S(t)U_0+\int_0^t S(t-s)\left[ \sum_{i=1}^{+\infty} k_i(s)B_iU(s-\tau_i(s))+F(U(s))\right] ds.
$$
Since F is Lipschitz and \eqref{dis1} holds, we have that 
$$
\begin{array}{l}
\displaystyle{ ||\Phi U(t)||_H\leq M e^{\omega t}||U_0||_H+M e^{\omega t} \int_0^t e^{-\omega s} \sum_{i=1}^{+\infty} |k_i(s)|\cdot||B_i U(s-\tau_i(s))||_H ds}\\
\hspace{1.8 cm}
\displaystyle{ +MLe^{\omega t}\int_0^t  e^{-\omega s}||U(s)||_H ds.}
\end{array}
$$
By change of variables $z=\varphi_i(s):=s-\tau_i(s)$ and using hypothesis \eqref{derivatadelay} on the derivative of time delays, we obtain
$$
\begin{array}{l}
\displaystyle{ ||\Phi U(t)||_H\leq M e^{\omega t} ||U_0||_H+Me^{\omega t} \sum_{i=1}^{+\infty} \int_{-\tau_i(0)}^{t-\tau_i(t)} e^{-\omega z} \frac{|k_i(\varphi_i^{-1}(z))|}{1-c_i} ||B_i U(z)||_H dz}\\
\hspace{2 cm}
\displaystyle{+MLe^{\omega t} \int_0^t e^{-\omega s} ||U(s)||_H ds.}
\end{array}
$$
Now, since the integrand function
$$
 e^{-\omega z}\frac{1}{1-c_i}|k_i(\varphi_i^{-1}(z))|\cdot||B_i U(z)||_H \geq 0
$$ 
for any $z\geq -\tau_*$ and $i\in\nat$, we obtain
$$
\begin{array}{l}
\displaystyle{||\Phi U(t)||_H\leq Me^{\omega t} ||U_0||_H+Me^{\omega t} \sum_{i=1}^{+\infty}\int_{-\tau_*}^0  e^{\omega \tau_*} \frac{|k_i(\varphi^{-1}_i (z))|}{1-c_i}||g_i(z)||_H dz}\\
\hspace{1.8 cm}
\displaystyle{+M e^{\omega t} \sum_{i=1}^{+\infty}\int_0^t e^{-\omega z} \frac{|k_i(\varphi^{-1}_i (z))|}{1-c_i}b_i ||U(z)||_H dz+MLe^{\omega t}\int_0^t e^{-\omega s} ||U(s)||_H ds}\\
\hspace{1.8 cm}
\displaystyle{ \leq M e^{\omega t} ||U_0||_H+Me^{\omega t} \sum_{i=1}^{+\infty} \int_{-\tau_*}^0 e^{\omega \tau_*} \frac{|k_i(\varphi^{-1}_i (z))|}{1-c_i}||g_i(z)||_H dz}\\
\hspace{1.8 cm}
\displaystyle{+M e^{\omega t}\sum_{i=1}^{+\infty} \int_0^t e^{(\omega'-\omega) z} \frac{|k_i(\varphi^{-1}_i (z))|}{1-c_i}b_i ||U||_Y dz+ ML e^{\omega t} \int_0^t e^{(\omega'-\omega)s} ||U||_Y ds}\\
\hspace{1.8 cm}
\displaystyle{ \leq M e^{\omega t} ||U_0||_H+Me^{\omega t} \sum_{i=1}^{+\infty}\int_{-\tau_*}^0  e^{\omega \tau_*} \frac{|k_i(\varphi^{-1}_i (z))|}{1-c_i}||g_i(z)||_H dz}\\
\hspace{1.8 cm}
\displaystyle{+M e^{\omega' t} ||U||_Y\sum_{i=1}^{+\infty} \int_0^t \frac{|k_i(\varphi^{-1}_i (z))|}{1-c_i}b_i  dz+ML||U||_Y \frac{ e^{\omega' t}}{\omega'-\omega} \left( 1-e^{-(\omega'-\omega)t}\right),}
\end{array}
$$
where in the last two inequalities we have simply used basic properties of exponential function and the definition of $||\cdot||_Y$. Hence, we obtain the following inequality:
$$
\begin{array}{l}
\displaystyle{ e^{-\omega' t}||\Phi U(t)||_H \leq M||U_0||_H +M e^{\omega \tau_*}\sum_{i=1}^{+\infty}\int_{-\tau_*}^0  \frac{|k_i(\varphi_i^{-1}(z))|}{1-c_i}||g_i(z)||_H dz}\\
\hspace{2.75 cm}
\displaystyle{+M||U||_Y  \sum_{i=1}^{+\infty} \int_0^tb_i \frac{|k_i(\varphi_i^{-1}(z))|}{1-c_i}dz+ML||U||_Y\frac{1}{\omega'-\omega}\left(1-e^{-(\omega'-\omega)t}\right).} 
\end{array}
$$
Therefore, using assumption \eqref{assumption}, the fact that $U\in Y$ and $\omega'>\omega$ yield
$$
||\Phi U(t)||_Y< +\infty.
$$
Hence, $\Phi U\in Y$. Now, we need to show that $\Phi$ is a contraction. To do so, let $u_1,u_2\in Y$. Then,
$$
\begin{array}{l}
\displaystyle{||\Phi u_1-\Phi u_2||_H \leq M\int_0^t e^{\omega (t-s)} \sum_{i=1}^{+\infty} |k_i(s)| \cdot ||B_i u_1(s-\tau_i(s))-B_iu_2(s-\tau_i(s))||_H ds}\\
\hspace{2.5 cm}
\displaystyle{ +ML\int_0^t e^{\omega(t-s)} ||u_1(s)-u_2(s)||_H ds.}
\end{array}
$$
As before, by using a change of variables $z=\varphi_i(s)=s-\tau_i(s)$, we obtain
$$
\begin{array}{l}
\displaystyle{||\Phi u_1-\Phi u_2||_H\leq Me^{\omega t} \sum_{i=1}^{+\infty}\int_{-\tau_*}^0 e^{-\omega z}  b_i \frac{|k_i(\varphi_i^{-1}(z))|}{1-c_i} ||u_1(z)-u_2(z)||_H dz}\\
\hspace{3 cm}
\displaystyle{+Me^{\omega t} \sum_{i=1}^{+\infty}\int_0^t  e^{-\omega z}  b_i \frac{|k_i(\varphi_i^{-1}(z))|}{1-c_i} ||u_1(z)-u_2(z)||_H dz}\\
\hspace{4 cm}
\displaystyle{+MLe^{\omega t}||u_1-u_2||_Y\int_0^t e^{-(\omega'-\omega)s}  ds}
\end{array}
$$
Again, from the definition of $||\cdot||_Y$ we have
$$
\begin{array}{l}
\displaystyle{||\Phi u_1-\Phi u_2||_H\leq Me^{\omega t} ||u_1-u_2||_Y \sum_{i=1}^{+\infty}\int_{-\tau_*}^0 e^{(\omega'-\omega) z} b_i \frac{|k_i(\varphi_i^{-1}(z))|}{1-c_i} dz}\\
\hspace{3 cm}
\displaystyle{+Me^{\omega t}||u_1-u_2||_Y\sum_{i=1}^{+\infty}\int_0^t e^{(\omega'-\omega)z} b_i \frac{|k_i(\varphi_i^{-1}(z))|}{1-c_i} dz}\\
\hspace{4 cm}
\displaystyle{+MLe^{\omega t}||u_1-u_2||_Y\frac{1}{\omega'-\omega} \left(e^{(\omega'-\omega)t}-1\right). }
\end{array}
$$
Then, using  $\omega ' > \omega$, we have that
$$
\begin{array}{l}
\displaystyle{ ||\Phi u_1-\Phi u_2||_Y \leq M \left( \tilde{M}+\frac{L}{\omega'-\omega}\right) ||u_1-u_2||_Y,}
\end{array}
$$
where
$$
\tilde{M}:=e^{\omega \tau_*}\sum_{i=1}^{+\infty} \int_{-\tau_*}^0 b_i \frac{|k_i(\varphi_i^{-1}(z))|}{1-c_i}dz+ \sum_{i=1}^{+\infty}\int_0^{+\infty} b_i \frac{|k_i(\varphi_i^{-1}(z))|}{1-c_i}dz.
$$
By assumption \eqref{assumption}, taking $\omega'>\omega$ such that 
$$
\tilde{M}+\frac{L}{\omega'-\omega}< \frac{1}{M},
$$
we have that $\Phi$ is a contraction on $Y$. Therefore, from Banach fixed point theorem we obtain the existence of a mild solution to \eqref{abstract} of the form \eqref{mild_solution}. In order to prove the uniqueness of the mild solution, let us suppose that there exist two solutions $u,v$ which satisfy \eqref{mild_solution}. Then,
$$
\begin{array}{l}
\displaystyle{ ||u(t)-v(t)||_H\leq Me^{\omega t}\int_0^t e^{-\omega s} \sum_{i=1}^{+\infty} |k_i(s)|\cdot ||B_i u(s-\tau_i(s))-B_i v(s-\tau_i(s))||_H ds }\\
\hspace{ 2.5 cm}
\displaystyle{ + ML e^{\omega t}\int_0^t e^{-\omega s} ||u(s)-v(s)||_H ds.}
\end{array}
$$
As before, by change of variables, and noticing that $B_i u(s)=B_i v(s)=g_i(s)$ for any $s\in I$, we have that
$$
\begin{array}{l}
\displaystyle{||u(t)-v(t)||_H\leq M e^{\omega t}\int_0^t e^{-\omega s} \sum_{i=1}^{+\infty} \frac{b_i}{1-c_i} |k_i(\varphi_i^{-1}(s))|\cdot ||u(s)-v(s)||_H ds}\\
\hspace{2.5 cm}
\displaystyle{+MLe^{\omega t} \int_0^t e^{-\omega s} ||u(s)-v(s)||_H ds.}
\end{array}
$$
Defining $\tilde{u}(t):=e^{-\omega t}||u(t)-v(t)||_H$ yields
$$
\begin{array}{l}
\displaystyle{ \tilde{u}(t)\leq M\int_0^t \beta (s) \tilde{u}(s) ds,}
\end{array}
$$
where
$$
\beta (s):= \sum_{i=1}^{+\infty} \frac{b_i}{1-c_i} |k_i(\varphi_i^{-1}(s))| + L.
$$
By direct Gronwall estimate, we can conclude that $||u(t)-v(t)||_H\equiv 0$ for any $t\geq 0$, which gives us a contradiction. Hence, the theorem is proved.
\end{proof}

\section{Exponential decay}\label{stabsec}
In this Section we will prove the exponential decay of solution to \eqref{abstract}, namely we prove Theorem \ref{decay}.
\begin{proof}[Proof of Theorem \ref{decay}]
From Duhamel's equation \eqref{mild_solution} we have that
$$
\begin{array}{l}
\displaystyle{||U(t)||_H\leq Me^{-\omega t} ||U_0||_H +Me^{-\omega t}\int_0^t e^{\omega s} \sum_{i=1}^{+\infty} |k_i(s)| ||B_iU(s-\tau_i(s))||_Hds}\\
\hspace{3 cm}
\displaystyle{+MLe^{-\omega t} \int_0^t e^{\omega s} ||U(s)||_H ds.}
\end{array}
$$
Hence, by usual change of variables, we obtain
$$
\begin{array}{l}
\displaystyle{ e^{\omega t}||U(t)||_H\leq M||U_0||_H+Me^{\omega \tau_*}\sum_{i=1}^{+\infty}\int_{-\tau_*}^0 \frac{1}{1-c_i} |k_i(\varphi_i^{-1}(s))|\cdot ||g_i(s)||_H ds}\\
\hspace{2.5 cm}
\displaystyle{+M\int_0^t e^{\omega s} \sum_{i=1}^{+\infty}\frac{b_i}{1-c_i} |k_i(\varphi_i^{-1}(s))| \cdot ||U(s)||_H ds+ML\int_0^t e^{\omega s} ||U(s)||_H ds.}
\end{array}
$$
By setting
$$
\tilde{U}(t):=e^{\omega t}||U(t)||_H
$$
for any $t\geq 0$ and
$$
\tilde{\alpha}:= M||U_0||_H+Me^{\omega \tau_*}\sum_{i=1}^{+\infty}\int_{-\tau_*}^0 \frac{1}{1-c_i} |k_i(\varphi_i^{-1}(s))|\cdot ||g_i(s)||_H ds,
$$
we obtain
$$
\tilde{U}(t)\leq \tilde{\alpha}+M\int_0^t \left[  \sum_{i=1}^{+\infty}\frac{b_i}{1-c_i} |k_i(\varphi_i^{-1}(s))| +L   \right] \tilde{U}(s) ds.
$$
Using Gronwall inequality yields
$$
||U(t)||_H\leq \tilde{\alpha} e^{M\int_0^t \beta(s)ds} e^{-(\omega-L)t},
$$
where 
$$
\beta(t):= \sum_{i=1}^{+\infty}\frac{b_i}{1-c_i} |k_i(\varphi_i^{-1}(t))| 
$$
for any $t\geq 0$. By assumption \eqref{assumption}, 
$$
\int_0^{+\infty} \beta (s)ds <\frac 1 M.
$$
Therefore,
$$
||U(t)||_H\leq \tilde{\alpha} e^{1-(\omega-L)t}, \quad t\geq 0.
$$
Since $L<\omega$, we obtain the thesis of the theorem.
\end{proof}

\section{Examples}\label{example}
In this Section we give some examples. We will show that the following models can be rewritten in the abstract form \eqref{abstract} in suitable Hilbert spaces, and, therefore, stability occurs under assumption \eqref{assumption}.
\subsection{The wave equation with memory and source term}\label{memory_wave_eq}
We present a semilinear wave equation with infinite memory damping and delay feedback. The effect of viscoelastic damping in wave equations with an extra delayed damping term has been analyzed in \cite{ANP} for the linear case and in \cite{PP} for the nonlinear one. We refer to \cite{ACS} for the undelayed system. Let $\Omega$ be a non-empty bounded subset of $\RR^n$, with $n\in\nat$ and denote its boundary by $\Gamma$. We suppose $\Gamma$ to be of class $C^2$. For any $i\in\nat$, we consider $\mathcal{O}_i\subset \Omega$ such that 
$$
\bigcup_{i=1}^{+\infty} \mathcal{O}_i= \Omega
$$
and $\mathcal{O}_i\cap \mathcal{O}_j=\emptyset$, whenever $i\neq j$. We consider the following system:
\begin{equation}\label{memory}
\begin{array}{l}
\displaystyle{u_{tt}(x,t)-\Delta u(x,t)+\int_0^{+\infty} \mu(s)\Delta u(x,t-s) ds}\\
\hspace{2 cm}
\displaystyle{+\sum_{i=1}^{+\infty} k_i(t)\chi_{\mathcal{O}_i}(x) u_t(t-\tau_i(t))=f(u(x,t)), \quad (x,t)\in \Omega \times (0,+\infty),}\\
\displaystyle{u(x,t)=0,  \quad (x,t)\in \Gamma\times (0,+\infty),}\\
\displaystyle{u(x,t)=u_0(x,t), \quad (x,t)\in \Omega\times (-\infty,0],}\\
\displaystyle{ u_t(x,0)=u_1(x), \quad x\in\Omega,}\\
\displaystyle{u_t(x,t)=g(x,t), \quad (x,t)\in\Omega\times I=\Omega \times\bigcup_{i=1}^{+\infty} [-\tau_i(0),0],}
\end{array}
\end{equation}
where, for any $i\in\nat$, $\tau_i(\cdot)$ are the time-dependent time delays, which satisfy hypothesis \eqref{derivatadelay}, $k_i(\cdot)\in L^1_{loc} ([0,+\infty))$, $f$ is a global Lipschitz continuous function of $u$ and $\mu(\cdot)$ is a locally absolutely continuous memory kernel which satisfies the following assumptions:
\begin{itemize}
\item[$(i)$] $\mu\in C^1(\RR^+)\cap L^1(\RR^+)$;
\item[$(ii)$] $\mu(0)=\mu_0>0$;
\item[$(iii)$]  $\int_0^{+\infty} \mu(s) ds =\tilde{\mu} <1$;
\item[$(iv)$] $\mu '(t)\leq -\delta \mu(t)$, for any $t\geq 0$ and for some $\delta >0$.
\end{itemize} 
As in Dafermos \cite{Dafermos}, we introduce the following auxiliar function:
\begin{equation}\label{eta}
\eta^t(x,s):=u(x,t)-u(x,t-s), \quad x\in\Omega, \ s,t\in (0,+\infty).
\end{equation}
Therefore, we can rewrite system \eqref{memory} in the following way
\begin{equation}
\label{memory+source_riscritta}
\begin{array}{l}
\displaystyle{ u_{tt}(x,t)-(1-\tilde{\mu})\Delta u(x,t)-\int_0^{+\infty} \mu(s)\Delta \eta^t(x,s) ds}\\
\hspace{3 cm}
\displaystyle{+\sum_{i=1}^N k_i(t)\chi_{\mathcal{O}_i}(x)u_t(x,t-\tau_i(t))=f(u(x,t)), \ (x,t)\in \Omega\times (0,+\infty),}\\
\displaystyle{\eta_t^t(x,s)=-\eta_s^t(x,s)+u_t(x,t), \ (x,t,s)\in \Omega\times (0,+\infty)\times (0,+\infty),}\\
\displaystyle{ u(x,t)=0, \ (x,t)\in \Gamma \times (0,+\infty),}\\
\displaystyle{ \eta^t(x,s)=0, \ (x,s) \in \Gamma \times (0,+\infty), \quad \text{for} \ t\geq 0,}\\
\displaystyle{ u(x,0)=u_0(x):=u_0(x,0), \ x\in \Omega,}\\
\displaystyle{ u_t(x,0)=u_1(x):=\frac{\partial u_0}{\partial t}(x,t)\Bigr| _{t=0},  \ x\in \Omega,}\\
\displaystyle{ \eta^0(x,s)=\eta_0(x,s):=u_0(x,0)-u_0(x,-s), \ (x,s)\in \Omega\times (0,+\infty),}\\
\displaystyle{ u_t(x,t)=g(x,t),  \ (x,t)\in \Omega\times I.}
\end{array}
\end{equation}
We consider the Hilbert space $L^2_\mu ((0,+\infty);H^1_0(\Omega))$ endowed with the inner product
$$
\langle \phi , \psi\rangle _{L^2_\mu ((0,+\infty);H^1_0(\Omega))}:= \int_{\Omega} \left( \int_0^{+\infty} \mu(s) \nabla \phi (x,s) \nabla \psi (x,s) ds \right) dx,
$$
and consider the Hilbert space
$$
H=H_0^1(\Omega)\times L^2(\Omega) \times L^2_\mu ((0,+\infty);H^1_0(\Omega)),
$$
equipped with the inner product
\begin{equation*}
\left\langle
\left (
\begin{array}{l}
u\\
v\\
w
\end{array}
\right ),
\left (
\begin{array}{l}
\tilde u\\
\tilde v\\
\tilde w
\end{array}
\right )
\right\rangle_H:=(1-\tilde{\mu}) \int_{\Omega} \nabla u \nabla \tilde u dx +\int_{\Omega} v\tilde v dx+ \int_{\Omega} \int_0^{+\infty} \mu(s) \nabla w \nabla \tilde w ds dx.
\end{equation*}
Setting $U=(u,u_t,\eta^t)$, we can rewrite \eqref{memory+source_riscritta} in the form \eqref{abstract}, where
\begin{equation}\label{Amemory}
A \begin{pmatrix}
u\\
v\\
w
\end{pmatrix}
=
\begin{pmatrix}
v\\
(1-\tilde{\mu})\Delta u+\int_0^{+\infty} \mu(s) \Delta w(s) ds \\
-w_s +v
\end{pmatrix},
\end{equation}
with domain
\begin{eqnarray*}
\mathcal{D(A)}&=&\{  (u,v,w)\in H_0^1(\Omega)\times H_0^1(\Omega)\times L^2_\mu ((0,+\infty);H^1_0(\Omega)) : \\
& & (1-\tilde{\mu})u +\int_0^{+\infty} \mu(s)w(s)ds \in H^2(\Omega)\cap H_0^1(\Omega), \ w_s\in L^2_\mu ((0,+\infty);H^1_0(\Omega))\} ,
\end{eqnarray*}
$B_i(u,v,\eta^t)^T := (0,\chi _{\mathcal{O}_i} v,0)^T$, for any $i\in\nat$, and $F(U(t))=(0,f(u(t)), 0)^T$. $A$ is exponentially stable (see e.g. \cite{Giorgi}). Therefore, if $L<\omega$ and assumption \eqref{assumption} is satisfied, then Theorem \ref{decay} holds. Hence, we obtain the stability result for system \eqref{memory}.

\subsection{The damped wave equation with source term}
As before, let $\Omega$ be a bounded domain of $\RR^n$, $n\in\nat$, with boundary $\partial\Omega$ of class $C^2$, and consider an open subset $\mathcal{O} \subset \Omega$, which satisfies the geometric control property in \cite{BLR}.  It can be seen, for instance, as the intersection of the domain $\Omega$ with an open neighborhood of the set (see \cite{Lions})
$$
\Gamma_0=\{ x\in \partial \Omega , \ m(x)\cdot \nu(x)>0\},
$$ 
where $m(x)=x-x_0, \ x_0\in\RR^n$. We take a countably infinite family of subsets $\{\mathcal{O}_i\}_{i\in\nat }$ of $\mathcal{O}$, such that
$$
\bigcup_{i=1}^{+\infty} \mathcal{O}_i=\mathcal{O}, \qquad \mathcal{O}_i\cap \mathcal{O}_j=\emptyset, \ i\neq j.
$$
Let $u:\Omega\times [0,+\infty)\to\RR$ be the solution to the following damped nonlinear wave equation:
\begin{equation}\label{wave_damped}
\begin{array}{l}
\displaystyle{u_{tt}(x,t) -\Delta u(x,t)+a\chi_{\mathcal{O}}(x) u_t(x,t)}\\
\hspace{3.2 cm}
\displaystyle{ +\sum_{i=1}^{+\infty} k_i(t)\chi_{\mathcal{O}_i}(x) u_t(x,t-\tau_i(t))=f(u(x,t)),\quad (x,t)\in \Omega\times (0,+\infty),}\\
\displaystyle{u(x,t)=0, \quad (x,t)\in\partial\Omega\times (0,+\infty),}\\
\displaystyle{u(x,0)=u_0(x), \ u_t(x,0)=u_1(x), \quad x\in\Omega,}\\
\displaystyle{u_t(x,s)=g(x,s), \quad (x,s)\in\Omega\times I.} 
\end{array}
\end{equation}
Here, $a$ is a positive constant, $\tau_i(\cdot)$ are the time-dependent time delays, for any $i\in\nat$, with $\tau_*$ defined as in \eqref{maxdelay}, and $k_i(\cdot)\in L^1_{loc}([0,+\infty))$, for any $i\in\nat$, while $f$ is a generic Lipschitz function of $u$.

We mention \cite{P2012} for the linear version of \eqref{wave_damped} (i.e. $f\equiv 0$) and \cite{KP} for the nonlinear one. If we define $U(t):=(u(t),u_t(t))$, we can rewrite system \eqref{wave_damped} in the abstract form \eqref{abstract}, with $H=H_0^1(\Omega)\times L^2(\Omega)$. The operator $A$ is defined as 
$$
A=
\begin{pmatrix}
0 & Id \\
\Delta &-a\chi_{\mathcal{O}}
\end{pmatrix},
$$
where $Id$ stands for the identity operator, while $B_i:H\to H$ are defined for any $i\in\nat$ as
$$
B_i \begin{pmatrix} u\\ v \end{pmatrix} = \begin{pmatrix} 0 \\ \chi_{\mathcal{O}_i} v \end{pmatrix} .
$$
We mention \cite{Zuazua, Komornik} in order to claim that $A$ generates a $C_0$-semigroup of contractions, which is exponentially stable. Therefore, we can apply Theorem \ref{decay} and exponential stability occurs under the main assumption \eqref{assumption}.

\subsection{The nonlinear plate equation with viscoelastic damping and source term}
As last example, we apply our theory of Sections \ref{Wellpos} and \ref{stabsec} to a nonlinear plate equation with a viscoelastic term and delay feedbacks. As above, we consider $\Omega$ a bounded domain of $\RR^n$, for $n\in\nat$, with boundary $\partial\Omega$ of class $C^2$. Let $\{\mathcal{O}_i\}_{i\in\nat}$ be a family of disjoint open sets of $\Omega$, which cover all the domain $\Omega$. Let us consider the following plate equation
\begin{equation}\label{plate}
\begin{array}{l}
\displaystyle{ u_{tt}(x,t)+\Delta^2(x,t)-\int_0^{+\infty}\mu(s)\Delta^2u(x,t-s)ds}\\
\hspace{3 cm}
\displaystyle{+\sum_{i=1}^{+\infty} k_i(t)\chi_{\mathcal{O}_i}(x)u_t(x,t-\tau_i(t))=f(u(x,t)),\quad (x,t)\in\Omega\times (0,+\infty)}\\
\displaystyle{ u(x,t)=\frac {\partial u}{\partial\nu }(x,t)=0, \ (x,t)\in\partial\Omega \times (0,+\infty),}\\
\displaystyle{ u(x,t)=u_0(x,t), \ (x,t)\in\Omega \times (-\infty,0],}\\
\displaystyle{ u_t(x,0)=u_1(x), \ x\in\Omega,} \\
\displaystyle{ u_t(x,t)=g(x,t), \ (x,t)\in\Omega \times I,}
\end{array}
\end{equation}
where $\tau_i(\cdot)$ are the time-dependent time delays, for any $i\in\nat$, $k_i(\cdot)\in L^1_{loc}([0,+\infty))$ are the damping coefficients and $\mu:(0,+\infty)\to (0,+\infty)$ is a locally absolutely continuous memory kernel which satisfies hypotheses $(i)-(iv)$ of Subsection \ref{memory_wave_eq}. Following \cite{Dafermos}, if we set $\eta^t$ as in \eqref{eta}, we can rewrite system \eqref{plate} in the following form
\begin{equation}
\label{plate2}
\begin{array}{l}
\displaystyle{ u_{tt}(x,t)+(1-\tilde{\mu})\Delta^2 u(x,t)+\int_0^{+\infty} \mu(s)\Delta^2 \eta^t(x,s) ds}\\
\hspace{3 cm}
\displaystyle{+\sum_{i=1}^{+\infty} k_i(t)\chi _{\mathcal{O}_i}(x) u_t(x,t-\tau_i(t))=f(u(x,t)), \ (x,t)\in \Omega\times (0,+\infty),}\\
\displaystyle{\eta_t^t(x,s)=-\eta_s^t(x,s)+u_t(x,t), \ (x,t,s)\in\Omega\times (0,+\infty)\times (0,+\infty),}\\
\displaystyle{ u(x,t)=\frac {\partial u}{\partial\nu }(x,t)= 0, \ (x,t)\in\partial\Omega \times (0,+\infty),}\\
\displaystyle{ \eta^t(x,s)=0, \ (x,s)\in \partial\Omega \times (0,+\infty), \quad \text{for} \ t\geq 0,}\\
\displaystyle{ u(x,0)=u_0(x):=u_0(x,0), \ x\in \Omega,}\\
\displaystyle{ u_t(x,0)=u_1(x):=\frac{\partial u_0}{\partial t}(x,t)\Bigr| _{t=0}, \ x\in \Omega,}\\
\displaystyle{ \eta^0(x,s)=\eta_0(x,s):=u_0(x,0)-u_0(x,-s), \ (x,s)\in \Omega\times (0,+\infty),}\\
\displaystyle{ u_t(x,t)=g(x,t), \ (x,t)\in\Omega\times I.}

\end{array}
\end{equation}

As before, if we consider $U(t)=(u(t),u_t(t))$, for any $t\geq 0$, we can rewrite system \eqref{plate} in the form \eqref{abstract}, where the operator $A$ is defined similarly to \eqref{Amemory}:
$$
A \begin{pmatrix}
u\\
v\\
w
\end{pmatrix}
=
\begin{pmatrix}
v\\
-(1-\tilde{\mu})\Delta^2 u-\int_0^{+\infty} \mu(s) \Delta^2 w(s) ds \\
-w_s +v
\end{pmatrix},
$$
$B_i=(0, \chi_{\mathcal{O}_i},0)^T$, for any $i\in\nat$ and $F(U(t))=(0,f(u(t)),0)^T$. Hence, if $f$ is Lipschitz continuous, we can apply Theorem \ref{decay}, which gives existence and uniqueness of a mild solution to \eqref{plate2} and its exponential stability, provided assumption \eqref{assumption} is satisfied.

\section*{Acknowledgments}
The research of the author is partially supported by the National GNAMPA (INdAM) Project 2020/2021 \emph{Buona positura, regolarità e controllo per alcune equazioni d’evoluzione}.

\end{document}